\documentclass{article}

\usepackage{times}
\usepackage{amsthm}
\usepackage{amsmath}
\usepackage{amssymb}
\usepackage{mathrsfs}
\usepackage{pb-diagram}
\usepackage{color}

\usepackage{hyperref}

\usepackage{graphicx}        % standard LaTeX graphics tool
\usepackage{pb-diagram}
\def\A{{\mathcal A}}

\def\th{\theta}
\def\Th{\Theta}

\def\T{\mathbb{T}}
\def\X{\mathcal X}
\def\Y{\mathcal Y}
\def\Z{\mathcal Z}
\def\H{\mathcal H}
\def\R{\mathbb{R}}
\def\S{\mathcal S}
\def\e{{\sf e}}
\def\g{{\mathfrak g}}

\def\({\left(}
\def\[{\left[}
\def\){\right)}
\def\]{\right]}

\def\bu{{\bullet}}

\def\G{{\sf G}}
\def\wG{\widehat{\sf{G}}}
\def\p{\parallel}
\def\<{\langle}
\def\>{\rangle}

\newtheorem{Theorem}{Theorem}[section]
\newtheorem{Remark}[Theorem]{Remark}
\newtheorem{Lemma}[Theorem]{Lemma}

\newtheorem{Corollary}[Theorem]{Corollary}
\newtheorem{Proposition}[Theorem]{Proposition}
\newtheorem{Definition}[Theorem]{Definition}

\begin{document}

\title{Quantization and Coorbit Spaces\\ for Nilpotent Groups}
%\titlerunning{Quantization and Coorbit Spaces for Nilpotent Groups}

\author{M. Mantoiu 
\medskip
}

\maketitle \vspace{-1cm}

\bigskip
\bigskip
{\bf Address}

\medskip
Departamento de Matem\'aticas, Universidad de Chile,

Las Palmeras 3425, Casilla 653, Santiago, Chile

\emph{E-mail:} mantoiu@uchile.cl

\bigskip

%\author{M. M\u antoiu and M. Ruzhansky \footnote{
%\textbf{2010 Mathematics Subject Classification: Primary 46L65, 47G30, Secondary 22D10, 22D25.}
%\newline
%\textbf{Key Words:}  locally compact group, nilpotent Lie group, noncommutative Plancherel theorem, pseudo-differential operator, $C^*$-algebra, dynamical system.}}

\maketitle

\begin{abstract}
We reconsider the quantization of symbols defined on the product between a nilpotent Lie algebra and its dual. To keep track of the non-commutative group background, the Lie algebra is endowed with the Baker-Campbell-Hausdorff product, making it via the exponential diffeomorphism a copy of its unique connected simply connected nilpotent Lie group. Using harmonic analysis tools, we emphasize the role of a Weyl system, of the associated Fourier-Wigner transformation and, at the level of symbols, of an important family of exponential functions. Such notions also serve to introduce a family of phase-space shifts. These are used to define and briefly study a new class of coorbit spaces of symbols and its relationship with coorbit spaces of vectors, defined via the Fourier-Wigner transform.
\end{abstract}

%---------------------------------------------------------------------------------------------------------
\section{Introduction}\label{induction}
%-----------------------------------------------------------------------------------------------------------

We discuss quantization on connected, simply connected nilpotent Lie groups $\G$ involving scalar-valued symbols. The main reason for which this is (at least formally) straightforward is the fact that the exponential $\exp:\g\to\G$ is a diffeomorphism and, under it, the Haar measures on $\G$ are proportional with the Lebesgue measure on the Lie algebra $\g$\,. Denoting by $\g^\sharp$ the dual of the Lie algebra, the symbols are complex functions defined on $\G\times\g^\sharp$ or, equivalently, on $\g\times\g^\sharp$. Of course, both these spaces can be seen as cotangent spaces, but we insist on the fact that the quantization is expected to be "global" (choosing charts for constructing the calculus is not needed and would be harmful) and that the group structure of $\G$ should play an important role. Another way to see the usefulness of nilpotence is to note that it allows a well-behaved Fourier transformation from functions or distributions on $\G$ to functions or distributions on $\g^\sharp$ and this Fourier transformation plays an important role in the pseudo-differential calculus.

\smallskip
Basically, if $a$ is a function on $\G\times\g^\sharp$ and $\varphi$ a function on $\G$\,, under favorable circumstances and with suitable interpretations, one is interested in
\begin{equation}\label{begin}
\big[Op(a)\varphi\big]\!(x):=\int_\G\int_{\g^\sharp}\!e^{i\<\log(xy^{-1})\mid \xi\>}a(x,\xi)\varphi(y)\,dyd\xi
\end{equation}
where, by definition, $\log:=\exp^{-1}\!:\G\to\g$\,. After suitable isomorphic compositions, this yields the equivalent form \eqref{garaci}, in which $\bu$ is the Backer-Campbell-Hausdorff composition on the Lie algebra, leading (via the exponential map) to a group isomorphism $(\G,\cdot)\cong(\g,\bu)$\,.

\smallskip
Although quite natural, the prescription \eqref{garaci} (or \eqref{begin}) has not been considered in such a general setting until recently, and the problem of defining and studying good H\"ormander-type symbol classes is a non-trivial challenge. Important articles have been dedicated to particular cases. Here "particular" very often means restricting to invariant (convolution) operators (formally, the function $a$ in \eqref{begin} only depends on $\xi$). It could also mean that the nilpotent group is two-step, or graded. Since we are not concerned here with the difficult problem of a H\"ormander-type calculus, we only cite \cite{CGGP,Dy,Glo1,Glo2,Glo3,Ma1,Ma2,Me,Mi,Mi1} without details. Let us mention, however, that in the cited articles of P. G\l owacki and D. Manchon, the invariant calculus is studied in depth, partly relying on the important previous work of R. Howe \cite{Ho,Ho1}. In \cite{MR1} the case of nilpotent groups with (generic) flat coadjoint orbit is treated, making a precise connection with the well-developed \cite{FR,FR1,MR,RT} operator-valued pseudo-differential calculus on $\G\times\wG$\,, where $\wG$ is the unitary dual of $\G$\,, i.e.\,the family of all equivalence classes of irreducible unitary Hilbert space representations of $\G$\,.

\smallskip
In the present paper, we mainly rely on Harmonic Analysis tools. We give a short description of its content.

\smallskip
{\bf Section \ref{bogart}} contains basic notions and notations concerning Hilbert space operators and nilpotent groups. A remarkable family of exponential functions on $\Xi:=\g\times\g^\sharp$ is introduced; it will play an important role subsequently.

\smallskip
{\bf Section \ref{firica}}. The basic object is {\it the Weyl system} $\big\{{\sf E}(X,\xi)\!\mid\!X\in\g\,,\,\xi\in\g^\sharp\big\}$, a family of unitary operators in $L^2(\g)$ mixing left translations associated to the group $(\g,\bu)$ with multiplication by imaginary exponentials with phase given by the duality between $\g$ and $\g^\sharp$. This family is very far from being a projective representation of the product group $(\g,\bu)\times\big(\g^\sharp,+\big)$ (that we denote by $\Xi$ and call {\it phase space}). This is why most of the technics developed in the literature do not apply automatically. It can be used to shift bounded operators $A\to{\sf E}(X,\xi)A{\sf E}(Y,\eta)^*$ in a useful way.  The "matrix coefficients" of the Weyl system defines {\it the Fourier-Wigner transformation}. It satisfies orthogonal relations and serves to introduce rigorously Berezin-type operators, cf. \cite{M}.

\smallskip
{\bf Section \ref{fifirin}}. We introduce the quantization ${\sf Op}$ and state its connection with the Weyl system, the Fourier-Wigner transformation and the Berezin quantization. Then we indicate the $^*$-algebraic laws on symbols that correspond to composition and adjunction of pseudo-differential operators. On the Schwartz space $\S(\Xi)$ one gets a Hilbert algebra structure, that can be used to perform various extensions of the laws by duality techniques, that will be useful below. In particular, one gets a (rather large, but not so explicit) {\it Moyal algebra} of symbols, quantized by operators that are continuous on the Schwartz space $\S(\g)$ and extend continuously on the dual $\S'(\g)$\,.

\smallskip
{\bf Section \ref{frikasse}}. Since the elements of the Weyl system are obtained by quantizing the special exponential function introduced in the first section, for pseudo-differential operators the shift $A\to{\sf E}(X,\xi)A{\sf E}(Y,\eta)^*$ is emulated by a similar one at the level of symbols, involving these exponentials and the intrinsic algebraic laws. We give the definition and the basic properties, that are interesting in their own right, are used in section \ref{frikamie} for constructing coorbit spaces of symbols, and will reappear in our subsequent study of a Beals-Bony commutator criterion on nilpotent groups. 

\smallskip
{\bf Section \ref{frikamie}} is dedicated to a tentative definition of coorbit spaces at two levels: (i) coorbit spaces of vectors, contained in $\S'(\g)$\,, and (ii) coorbit spaces of symbols, contained in $\S'(\Xi)$\,. We declare from the very beginning that the treatment is incomplete from many points of view. Although there is a lot of group theory around, one does not follow the orbits of some (usual or projective) group representation. At both levels, one uses isometric linear mappings, labeled by fixed functions (windows), sending functions (or distributions) on the space we are interested in (here $\g$\,, respectively $\Xi$), to functions or distributions on larger spaces ($\Xi$ or $\Xi\times\Xi$, respectively). Then one selects for the coorbit space elements that have a certain behavior under the isometry (belonging to a given subspace, or having a finite given norm).

\smallskip
For the first level we use the Fourier-Wigner transform, naturally depending on two vectors, fixing one of them as a window and measuring the dependence of the other one. For the particular case of the Heisenberg group, it would be interesting to compare the outcome with the constructions of \cite{FRR}, in which, a priori, the point of view is different. Doubling the number of variables, such a procedure would probably also work well for (ii), and this is roughly what is done in the Abelian case $\G=\R^n$ to define coorbit spaces of symbols.

\smallskip
However, we adopt another strategy. Adapting some abstract ideas from \cite{M0}, the isometry we use can be found in Definition \ref{minte} and it relies on the previously defined phase-space shifts of the symbol, coupled by duality with the chosen window. To advocate this choice, we put into evidence two properties of our isometry that can be obtained quite easily: 
\begin{enumerate}
\item For a self-adjoint idempotent window, it is a $^*$-algebra monomorphism (cf. Proposition \ref{fulminant}). The consequence (Corollary \ref{infekt}) is the fact that starting with an algebra of kernels defined on $\Xi\times\Xi$\,, stable under the natural kernel adjoint, the corresponding coorbit space is a $^*$-algebra of symbols with respect to the intrinsic symbol composition and adjoint (so, by quantization, one gets $^*$-algebras of pseudo-differential operators).
\item For a good correlation of the chosen windows, there is a formula \eqref{formmula} relying the two types of isometries, the ${\sf Op}$-calculus and the calculus of integral operators. This allows in Theorem \ref{margisnitele} to state roughly that pseudo-differential operators with symbols in certain coorbit spaces are bounded between certain coorbit spaces of vectors. Actually, this is merely stated in a weaker form making use of coorbit norms on Schwartz spaces.
\end{enumerate}

It is obvious that much more effort is needed to transform this sketchy treatment into a theory. Remarks \ref{critica} and \ref{conkrete} can be read as a self-criticism. To be brief, let us say that the abstract part is still incomplete, while the concrete part misses completely. Hopefully, there will be some progress in a future publication. It is not yet clear to us how far one can go, since the notions, although quite elementary, have complicated explicit expressions. Probably particular classes of nilpotent groups should be considered first.

\smallskip
It is clear that the coorbit theory part of this paper relies on many previous contributions of many authors. The number of interesting articles belonging to Time Frequency Analysis and treating modulation or other coorbit spaces on various mathematical structures and from various points of view is huge. Even if one restricts to classics and to those papers involving pseudo-differential operators, it is not the place here to sketch a history or at least to cite "most" of the references. Not forgetting to mention the central role played by H. Feichtinger and K. Gr\"ochenig, we make a selection of references \cite{CGNR,CNR,CR,CTW,Fe,Fe1,FG,Gr,Gr1,Gr2,Gr3,GH,GRo,GS,HTW,Sj,To,To1} that inspired or are related to the last section of the present article.

%---------------------------------------------------------------------------------------------------------
\section{Framework}\label{bogart}
%-----------------------------------------------------------------------------------------------------------

{\bf Conventions.}
The scalar products in a Hilbert space are linear in the first variable. For a given (complex, separable) Hilbert space $\H$\,, one denotes by $\mathbb B(\H)$ the $C^*$-algebra of all linear bounded operators in $\H$, by $\mathbb K(\H)$ the closed  bi-sided $^*$-ideal of all compact operators  and by $\mathbb B^2(\H)$ the bi-sided $^*$-ideal of all Hilbert-Schmidt operators. The group of unitary operators in $\H$ is denoted by $\mathbb U(\H)$\,. If $\mathcal F,\mathcal G$ are locally convex spaces, one sets $\mathbb L(\mathcal F,\mathcal G)$ for the space of linear continuous operators $T:\mathcal F\to\mathcal G$\,. We admit the abbreviation $\mathbb L(\mathcal F,\mathcal F)=:\mathbb L(\mathcal F)$\,.

\smallskip
Let $\G$ be a connected simply connected nilpotent Lie group with unit $\e$\,, center ${\sf Z}$\,, bi-invariant Haar measure $dx$ and unitary dual $\wG$\,. 
Let $\mathfrak g$ be the Lie algebra of $\G$ with center $\mathfrak z={\sf Lie}({\sf Z})$ and $\mathfrak g^\sharp$ its dual. If $X\in\g$ and $\xi\in\g^\sharp$ we set $\<X\!\mid\!\xi\>:=\xi(X)$\,. We also denote by $\exp:\mathfrak g\to\G$ the exponential map, which is a diffeomorphism. Its inverse is denoted by $\log:\G\rightarrow\mathfrak g$\,. Under these diffeomorphisms the Haar measure on $\G$ corresponds to a Haar measure $dX$ on $\g$ (normalized accordingly). It then follows that $L^p(\G)$ is isomorphic to $L^p(\g)$\,. 
The Schwartz spaces $\S(\G)$ and $\S(\g)$ are defined as in \cite[A.2]{CG}; they are isomorphic Fr\'echet spaces.

\begin{Remark}\label{gruzav}
For $X,Y\in\mathfrak g$ we set 
\begin{equation*}\label{diosdado}
X\bullet Y:=\log[\exp(X)\exp(Y)]\,.
\end{equation*}
It is a group composition law on $\mathfrak g$\,, given by a polynomial expression in $X,Y$ (the Baker-Campbel-Hausdorff formula). The unit element is $0$ and $X^{\bullet}\equiv -X$ is the inverse of $X$ with respect to $\bullet$\,. One has in fact
\begin{equation}\label{diostraso}
X\bullet Y=X+Y+\frac{1}{2}[X,Y]+\frac{1}{12}[X,[X,Y]]+\frac{1}{12}[Y,[Y,X]]+\dots\equiv X+Y+R(X,Y)\,,
\end{equation}
where, by nilpotency, the sum is finite. It seemed to us easier to work on the group $(\g,\bu)$\,, but transferring all the formalism to its isomorphic version $(\G,\cdot)$ is an obvious task.
\end{Remark}

{\it The adjoint action} \cite{CG} is 
\begin{equation*}\label{adji}
{\sf Ad}\colon\G\times\g\to\g\,,\quad {\sf Ad}_x(Y):=\frac{d}{dt}\Big\vert_{t=0}\big[x\exp(tY)x^{-1})\big]
\end{equation*}
and {\it the coadjoint action} of $\G$ is 
\begin{equation*}\label{coadji}
 {\sf Ad}^\sharp:\G\times\g^\sharp\to\g^\sharp, \quad (x,\eta)\mapsto {\sf Ad}^\sharp_x(\eta):=\eta\circ{\sf Ad}_{x^{-1}}. 
\end{equation*}
Translating to the Lie algebra, one gets
\begin{equation}\label{faras}
\widetilde{\sf Ad}^\sharp:\g\times\g^\sharp\to\g^\sharp, \quad (X,\eta)\mapsto \widetilde{\sf Ad}^\sharp_X(\eta):={\sf Ad}^\sharp_{\exp X}(\eta). 
\end{equation}

One has {\it the left and the right unitary representations} ${\sf L,R}:(\g,\bu)\to\mathbb U\big[L^2(\g)\big]$\,, defined by
\begin{equation*}\label{leftright}
\big[{\sf L}_Z(u)\big](X):=u\big([-Z]\bu X\big)\,,\quad\big[{\sf R}_Z(u)\big](X):=u(X\bu Z)\,.
\end{equation*}

We call (somehow inappropriately) {\it phase space} the direct product non-commutative group $\,(\Xi,\circ):=(\g,\bu)\times\big(\g^\sharp,+\big)$\,.

\begin{Definition}\label{diversitate}
For every $(Z,\zeta)\in\Xi$ we define $\,\varepsilon_{(Z,\zeta)}:\Xi\to\T$\,, $\,\varepsilon_Z:\g^\sharp\to\T$\,, $\,\varepsilon_\zeta:\g\to\T$ by
\begin{equation*}\label{treiga}
\varepsilon_{(Z,\zeta)}(X,\xi):=e^{i\<X\mid\zeta\>}e^{-i\<Z\mid\xi\>}\,,\quad\varepsilon_Z:=\varepsilon_{(Z,0)}\,,\quad\varepsilon_\zeta:=\varepsilon_{(0,\zeta)}\,.
\end{equation*}
\end{Definition}

These functions will play an important role in quantization. The proof of the following lemma consists in simple calculations; for \eqref{banane} one needs \eqref{diostraso}.

\begin{Lemma}\label{sitemic}
For all the values of the parameters one has 
\begin{equation*}\label{otita}
\varepsilon_{(Z,\zeta)}=\varepsilon_Z\varepsilon_\zeta\equiv\varepsilon_\zeta\otimes\varepsilon_Z\,,
\end{equation*}  
\begin{equation*}\label{galci}
\varepsilon_{(Z_1+Z_2,\zeta_1+\zeta_2)}=\varepsilon_{(Z_1,\zeta_1)}\varepsilon_{(Z_2,\zeta_2)}\,,
\end{equation*}
\begin{equation}\label{banane}
\varepsilon_{(Z,\zeta)}(X\bu Y,\xi+\zeta)=\varepsilon_{(Z,\zeta)}(X,\xi)\varepsilon_{(Z,\zeta)}(Y,\eta)e^{i\<R(X,Y)\mid\zeta\>}\,.
\end{equation}
\end{Lemma}

One has a linear topological isomorphism $\,\mathcal F:\mathcal S(\g)\rightarrow\mathcal S(\g^\sharp)$\,, {\it the Fourier transformation}, given by  
\begin{equation*}\label{qlata}
\big(\mathcal F u\big)(\xi)=\int_{\g}e^{-i\<X\mid \xi\>}u(X)dX=\int_\g \overline{\varepsilon_\xi(X)}\,u(X)dX\,.
\end{equation*}
For a unique good choice of the Haar measure $d\xi$ on $\g^\sharp$, the inverse is 
\begin{equation*}\label{qluta}
\big(\mathcal F^{-1}\mathfrak u\big)(X)=\int_{\g^\sharp}\!e^{i\<X\mid \xi\>}\mathfrak u(\xi)d\xi=\int_{\g^\sharp}\!\varepsilon_\xi(X)\,\mathfrak u(\xi)d\xi\,,
\end{equation*}
the transformation $\mathcal F$ is unitary from $L^2(\g;dX)$ to $L^2(\g^\sharp;d\xi)$ and extends to a topological isomorphism $\,\mathcal F:\mathcal S'(\g)\rightarrow\mathcal S'(\g^\sharp)$\,.
We are also going to use the total Fourier transformation $f\to\widetilde f$ given by
\begin{equation}\label{tutala}
\widetilde f(Z,\zeta):=\int_\g\int_{\g^\sharp} e^{-i\<Y|\zeta\>}e^{i\<Z|\eta\>}f(Y,\eta)dYd\eta\,.
\end{equation} 

\begin{Lemma}\label{ciulama}
For every $f,g\in\S(\Xi)$ one has
\begin{equation}\label{vormula}
\int_\Xi\<f,\varepsilon_\X\>_{(\Xi)}\<\varepsilon_\X,g\>_{(\Xi)}d\X=\<f,g\>_{(\Xi)}\,.
\end{equation}
\end{Lemma}

\begin{proof}
First one notes that $\<f,\varepsilon_\X\>_{(\Xi)}=\widetilde f(\X)$ and then invokes Plancherel's Theorem. 
\end{proof}

%-------------------------------------------------------------------------------------------------------
\section{Weyl systems, the Fourier-Wigner transform}\label{firica}
%----------------------------------------------------------------------------------------------------

\begin{Definition}\label{sigmund}
For any $(Z,\zeta)\in\g\times\g^\sharp=\Xi$ one defines a unitary operator ${\sf E}(Z,\zeta)$ in $L^2(\g)$ by
\begin{equation*}\label{friedar}
\[{\sf E}(Z,\zeta)u\]\!(X):=e^{i\<X\mid \zeta\>}u([-Z]\bu X)\,,
\end{equation*}
with adjoint
\begin{equation*}\label{frigider}
\[{\sf E}(Z,\zeta)^*u\]\!(Y)=e^{-i\<Z\bu Y\mid \zeta\>}u(Z\bu Y)\,.
\end{equation*}
\end{Definition}

This extends the notion of {\it Weyl system} (or {\it time-frequency shifts}) from the case $\G=\R^n$.
Since the composition law $\bu$ is polynomial, these operators also act as isomorphisms of the Schwartz space $\S(\g)$ and can be extended to isomorphisms of the space $\S'(\g)$ of tempered distributions. We are going to see below how they fit in the pseudo-differential calculus.

\begin{Lemma}\label{calcul}
Denote by ${\sf Mult}(\phi)$ the operator of multiplication by the function $\phi$\,. For $(Z,\zeta),(Y,\eta)\in\Xi$ one has
\begin{equation*}\label{fulppe}
{\sf E}(Z,\zeta)\,{\sf E}(Y,\eta)={\sf Mult}\big(\Upsilon\big[(Z,\zeta),(Y,\eta);\cdot\big]\big)\,{\sf E}(Z\bu Y,\zeta+\eta)\,,
\end{equation*}
where 
\begin{equation*}\label{sireata}
\Upsilon\big[(Z,\zeta),(Y,\eta);X\big]=\exp\big\{i\<\,[-Z]\bu X-X)\!\mid \!\eta\,\>\big\}\,.
\end{equation*}
\end{Lemma}

This follows from a direct calculation. The map ${\sf E}$ is not even a projective representation of the group $\Xi$\,, so standard tools in coorbit theory relying on group representations will not be available.

\smallskip
One also sets 
\begin{equation*}\label{makela}
{\sf E}_Z:={\sf E}(Z,0)\equiv{\sf L}_Z\,,\quad{\sf E}_\zeta:={\sf E}(0,\zeta)\equiv{\sf M}_\zeta={\sf Mult}(\varepsilon_\zeta)\,.
\end{equation*} 
Note the "multiplication relations"
\begin{equation*}\label{CCR}
{\sf L}_Y{\sf L}_Z={\sf L}_{Y\bu Z}\,,\quad {\sf M}_\eta {\sf M}_\zeta={\sf M}_{\eta+\zeta}\,,\quad {\sf L}_Z {\sf M}_\zeta=e^{i\<Z^{-1}\bu(\cdot)-(\cdot)\mid\zeta\>}{\sf M}_\zeta {\sf L}_Z\,,
\end{equation*}
that follow from Lemma \ref{calcul} or are shown directly. So one has (strongly continuous) unitary representation 
\begin{equation*}\label{celedoua}
{\sf M}:(\g^\sharp,+)\to\mathbb U\big[L^2(\g)\big]\quad{\rm and}\quad{\sf L}:(\g,\bu)\to\mathbb U\big[L^2(\g)\big]
\end{equation*} 
that do not commute to each other. 

\begin{Definition}\label{niulaif}
For any $\Y:=(Y,\eta),\,\mathcal Z:=(Z,\zeta)\in\Xi:=\g\times\g^\sharp$ we define the linear contraction
\begin{equation*}\label{automo}
\Th_{\Y,\mathcal Z}:\mathbb B(\H)\to\mathbb B(\H)\,,\quad\Th_{\Y,\Z}(A):={\sf E}(\Y)A\,{\sf E}(\Z)^*={\sf M}_\eta {\sf L}_Y A\,{\sf L}_{-Z}{\sf M}_{-\zeta} \,.
\end{equation*}
\end{Definition}

In particular, $\Th_{\Y,\mathcal Y}$ is an automorphism of the $C^*$-algebra $\mathbb B(\H)$\,. There are no simple group properties of the family.

\begin{Remark}\label{besides}
As said above, besides being unitary operators in $L^2(\g)$\,, the elements ${\sf E}(\Z)$ of the Weyl system can also be seen as isomorphisms of $\S(\g)$ or of its dual. Therefore $\Th_{\Y,\Z}$ also acts on $\mathbb L[\S(\g)]\,,\,\mathbb L[\S'(\g)]\,,\mathbb L[\S(\g),\S'(\g)]$ and $\mathbb L[\S'(\g),\S(\g)]$\,.
\end{Remark}

\begin{Definition}\label{wigner}
For $\,u,v\in\H:=L^2(\g)$ one sets $\mathcal E_{u,v}\equiv\mathcal E_{u\otimes v}:\g\times\g^\sharp\to\mathbb C$ by
\begin{equation*}\label{vigner}
\mathcal E_{u,v}(Z,\zeta):=\<{\sf E}(Z,\zeta)u,v\>_\H=\int_\g e^{i\<Y\mid\zeta\>}u([-Z]\bu Y)\overline{v(Y)}dY.
\end{equation*}
and call it {\rm the Fourier-Wigner transform}.
\end{Definition}

\begin{Lemma}\label{startortog}
The Fourier-Wigner transform extends to a unitary map 
\begin{equation*}\label{FW}
\mathcal E\!:\H\otimes\overline\H\cong L^2(\g\times\g)\to L^2(\Xi)\,.
\end{equation*}
It also defines isomorphisms
\begin{equation*}\label{theizo}
\mathcal E\!:\S(\g)\,\overline\otimes\,\S(\g)\cong\S(\g\times\g)\to \S(\Xi)\,,\quad\mathcal E\!:\S'(\g)\,\overline\otimes\,\S'(\g)\cong\S'(\g\times\g)\to \S'(\Xi)\,.
\end{equation*}
\end{Lemma}

\begin{proof}
It is composed of a partial Fourier transformation and a unitary change of variables, that is also $\S$-compatible. We denoted by $\overline{\otimes}$ the completed projective tensor product, but we recall that $\S(\g)$ is nuclear.
\end{proof}

In particular, one has {\it the orthogonality relations}:
\begin{equation}\label{orthog}
\big\<\mathcal E_{u,v},\mathcal E_{u'\!,v'}\big\>_{L^2(\Xi)}=\<u,u'\>_\H\,\<v',v\>_\H\,.
\end{equation}

From now on, we are going to use the notation $\<\cdot,\cdot\>_{(\Xi)}$ both for the scalar product in $L^2(\Xi)$ and for the duality between the Schwartz space on $\Xi$ and the space of temperate distributions.
Similarly for $\<\cdot,\cdot\>_{(\g)}$\,.

\begin{Remark}\label{uwe}
In \cite{M}, the Berezin-Toeplitz calculus for suitable symbols $h:\G\times\g^\sharp\to\mathbb C$ has been introduced and studied. As in the present article, $\G$ was a connected, simply connected nilpotent Lie group with Lie algebra $\g^\sharp$, but the Berezin operators act in $L^2(\G)$ or $\S(\G)$\,. By suitably composing with the diffeomorphism $\exp:\g\to\G$\, (both at the level of vectors and symbols) this may be recast in the present setting. For convenience of the reader, we indicate the basic definition, making use of the objects ${\sf E}$ and $\mathcal E$ introduced above. The full treatment in \cite{M} can easily transported on $\g\times\g^\sharp$.

\smallskip
Let $\,w\in \S(\g)$ be a normalized vector (it may also be chosen in $L^2(\g)$). We define in $L^2(\g)$
\begin{equation}\label{albina}
{\sf Ber}_w(h):=\int_\g\int_{\g^\sharp}\!h(X,\xi)\,\<\cdot,{\sf E}(X,\xi)^*w\>{\sf E}(X,\xi)^*w\,dXd\xi\,.
\end{equation}
The rigorous definition is in weak sense: for any $u,v\in L^2(\g)$ one has
\begin{align}
\big\<{\sf Ber}_w(h)u,v\big\>_{(\g)}:=&\int_\g\int_{\g^\sharp}\!h(X,\xi)\big\<u,{\sf E}(X,\xi)^*w\big\>\big\<{\sf E}(X,\xi)^*w,v\big\>\,\,dXd\xi\\
=&\int_\g\int_{\g^\sharp}h(X,\xi)\,\mathcal E_{u,w}(X,\xi)\,\overline{\mathcal E_{v,w}(X,\xi)}\,dXd\xi\\
=&\,\big\<h,\overline{\mathcal E_{u,w}}\,\mathcal E_{v,w}\big\>_{(\Xi)}\,.
\end{align}
This last expression and the properties of the Fourier-Wigner transform allow various interpretations of this formula, under various conditions on $u,v,w,h$\,.
\end{Remark}

%-------------------------------------------------------------------------------------------------------
\section{Pseudo-differential operators}\label{fifirin}
%-------------------------------------------------------------------------------------------------------

One has the quantizations of the "phase space" $\mathfrak g\times\mathfrak g^\sharp\ni(X,\xi)$\: 

\begin{align}\label{garaci}
\begin{split}
&{\sf Op}:L^2(\mathfrak g\times\mathfrak g^\sharp)\to\mathbb B^2\big[L^2(\g)\big]\,,\\
\big[&{\sf Op}(f)u\big](X)=\int_\g\int_{\g^\sharp}\!e^{i\<X\bu(-Y)\mid \xi\>}f\!\left(X,\xi\right)u(Y)\,dYd\xi\,.
\end{split}
\end{align}

\begin{Remark}\label{atreia}
{\rm Examining the kernel of ${\sf Op}(f)$\,, one easily sees that ${\sf Op}:L^2(\Xi)\to\mathbb B^2\big[L^2(\g)\big]$ is indeed an isomorphism. For similar reasons, by restriction or extension, one also has topological linear isomorphisms
\begin{equation*}\label{treiaka}
{\sf Op}:\S(\Xi)\overset{\sim}{\longrightarrow}\mathbb L\big[\S'(\g),\S(\g)\big]\,,\quad
{\sf Op}:\S'(\Xi)\overset{\sim}{\longrightarrow}\mathbb L\big[\S(\g),\S'(\g)\big]\,.
\end{equation*}
}
\end{Remark}

One may justify (at least heuristically) formula \eqref{garaci} in various ways: 

\begin{itemize}
\item
One could start with a canonical dynamical system, built over the left action of $(\g,\bullet)$ on itself and then raised to a $C^*$-action on function defined on $\g$\,. To such a data, there is a canonical construction of a $C^*$-algebra (the crossed product) and of a "Schr\"odinger representation" in $\H:=L^2(\g)$\,. The calculus ${\sf Op}$ is then obtained from this Schr\"odinger representation by composing with a partial Fourier transformation. For details we refer to \cite{MR}.
\item
In terms of the Weyl system $\big\{{\sf E}(Z,\zeta)\!\mid\!(Z,\zeta)\in\Xi\big\}$ and the total Fourier transformation $f\to\widetilde f$\,, one can write 
\begin{equation*}\label{stephann}
{\sf Op}(f)=\int_{\g}\!\int_{\g^\sharp}\widetilde{f}(Z,\zeta){\sf E}(Z,\zeta)\,dZd\zeta\,.
\end{equation*}
\item
In the simple Abelian case $\G\equiv\g=\R^n$ one has $X\bu(-Y)=X-Y$ and \eqref{garaci} boils down to the Kohn-Nirenberg quantization.
\end{itemize}

Actually, in terms of the Fourier-Wigner transform, one can write
\begin{equation}\label{napastta}
\<{\sf Op}(f)u,v\>_{(\g)}=\big\<\widetilde f,\overline{\mathcal E_{u,v}}\big\>_{(\Xi)}\,,
\end{equation}
allowing various types of ingredients $u,v,f$, having in view the properties of the transformation $\mathcal E$ and of the dualities. By using Plancherel's Theorem one could rewrite \eqref{napastta} as
\begin{equation}\label{napasttaka}
\<{\sf Op}(f)u,v\>_{(\g)}=\big\<f,\mathcal W_{u,v}\big\>_{(\Xi)}\,,
\end{equation}
and $(u,v)\to\mathcal W_{u,v}$ could be called {\it the Wigner transformation}.

\smallskip
It is easy to prove the next result:

\begin{Proposition}\label{cateva}
\begin{enumerate}
\item[(i)]
One has
\begin{equation*}\label{stephannen}
{\sf E}(Z,\zeta)={\sf Op}\big(\varepsilon_{(Z,\zeta)}\big)\,,\quad\forall\,(Z,\zeta)\in\Xi\,.
\end{equation*}
In particular ${\sf L}_Z={\sf Op}(\varepsilon_Z)$ and ${\sf M}_\zeta={\sf Op}(\varepsilon_\zeta)$\,.
\item[(ii)]
One has
\begin{equation*}\label{amandoua}
{\sf Op}(\phi\otimes\psi)={\sf Mult}(\phi){\sf Conv}_L(\mathcal F^{-1}\psi)\,,
\end{equation*}
the product between a multiplication operator and a left convolution operator (that is right invariant). Particular cases:
\begin{equation*}\label{farfurii}
f(X,\xi):=\phi(X)\ \Longrightarrow\ {\sf Op}(f)u=\phi u\,,
\end{equation*}
\begin{equation*}\label{tacamuri}
f(X,\xi):=\psi(\xi)\ \Longrightarrow\ {\sf Op}(f)u=(\mathcal F^{-1}\psi)\star u\,.
\end{equation*}
\item[(iii)]
The rank one operator $\<\cdot,v\>u$ coincides with ${\sf Op}\big(\mathcal E_{u,v}\big)$\,.
\end{enumerate}
\end{Proposition}

\begin{Remark}\label{fricca}
In \cite[Sect.\,6]{M}, a connection has been established between pseudo differential and Berezin-type operators with symbols defined on $\G\times\g^\sharp$. By properly composing with the exponential diffeomorphism, one lands in our framework and finds that the Berezin operator ${\sf Ber}_w(h)$ given in \eqref{albina} is an operator of the form \eqref{garaci}, with
\begin{align}
f(X,\xi):=\int_\g\int_\g\int_{\g^\sharp}& e^{-i\<Y\mid\xi\>}\,e^{i\<\log(Z\bullet[-Y]\bullet X)-Z\bullet X\mid\zeta\>}\\
&h(Z,\zeta)\,w(Z\bullet X)\,\overline{w(Z\bullet[-Y]\bullet X)}\,dY dZ d\zeta\,.
\end{align}
\end{Remark}

We treat now the intrinsic algebraic structure on symbols. The pseudo-differential operator \eqref{garaci} with symbol $f$ is an integral operator with kernel $\,{\sf Ker}_{f}:\g\times\g\rightarrow\mathbb C$ given by
\begin{equation}\label{acrisor}
{\sf Ker}_{f}(X,Y)=\int_{\g^\sharp}\!e^{i\<X\bu[-Y]\mid \xi\>}f(X,\xi)d\xi=\big[\big({\rm id}\otimes\mathcal F^{-1}\big)f\big]\big(X,X\bu[-Y]\big)\,.
\end{equation}
Inverting, the symbol may be recuperated from the kernel by means of the formula
\begin{equation}\label{recup}
f(X,\xi)=\int_\g e^{-i\<Y\mid\xi\>}{\sf Ker}_f\big(X,[-Y]\bu X\big)dY.
\end{equation}

\begin{Proposition}\label{algstruct}
\begin{enumerate}
\item[(i)]
The symbol $f\#g$ of the product ${\sf Op}(f){\sf Op}(g)$ is
\begin{align}\label{roduct}
(f\#g)(X,\xi)=\int_\g\!\int_\g \int_{\g^\sharp}\!\int_{\g^\sharp}&\overline{\varepsilon_\xi(Y)}\,\varepsilon_\eta(X\bu[-Z])\,\varepsilon_\zeta(Z\bu[-X]\bu Y)\\
&f(X,\eta)g(Z,\zeta)dYdZd\eta d\zeta\,.
\end{align}
\item[(ii)]
The symbol $f^\#$ of the adjoint ${\sf Op}(f)^*$ is
\begin{equation}\label{adjunct}
f^\#(X,\xi)=\int_\g\!\int_{\g^\sharp}\! e^{i\<Y\mid\eta-\xi\>}\,\overline{f([-Y]\bu X,\eta)}\,dYd\eta\,.
\end{equation}
\end{enumerate}
In particular $(\phi\otimes 1)^\#=\overline{\phi}\otimes 1$ and $(1\otimes\psi)^\#=1\otimes\overline{\psi}$\,.
\end{Proposition}

\begin{proof}
(i) One computes
\begin{align}
(f\#g)(X,\xi)&=\big({\sf Ker}^{-1}\big[{\sf Ker}_f\circ{\sf Ker}_g\big]\big)(X,\xi)\\
&=\int_\g e^{-i\<Y\mid\xi\>}\big({\sf Ker}_f\circ{\sf Ker}_g\big)\big(X,[-Y]\bu X\big)dY\\
&=\int_\g\!\int_\g e^{-i\<Y\mid\xi\>}{\sf Ker}_f(X,Z){\sf Ker}_g\big(Z,[-Y]\bu X\big)dYdZ\\
&=\int_\g\!\int_\g \int_{\g^\sharp}\!\int_{\g^\sharp}e^{-i\<Y\mid\xi\>}e^{i\<X\bu[-Z]\mid \eta\>}e^{i\<Z\bu[-X]\bu Y\mid \zeta\>}f(X,\eta)g(Z,\zeta)dYdZd\eta d\zeta\,.
\end{align}

(ii) If $K$ is the kernel of an integral operator, the kernel of the adjoint is given by $K^\square(X,Y):=\overline{K(Y,X)}$\,. Hence, by \eqref{acrisor} and \eqref{recup}
\begin{align}
f^\#(X,\xi)&=\int_\G e^{-i\<Y\mid\xi\>}{\sf Ker}_{f^\#}\big(X,[-Y]\bu X\big)dY\\
&=\int_\G e^{-i\<Y\mid\xi\>}{\sf Ker}_{f}^\square\big(X,[-Y]\bu X\big)dY\\
&=\int_\G e^{-i\<Y\mid\xi\>}\overline{{\sf Ker}_{f}\big([-Y]\bu X,X\big)}dY\\
&=\int_\G e^{-i\<Y\mid\xi\>}\overline{\int_{\g^\sharp}\!e^{-i\<Y\mid \eta\>}f([-Y]\bu X,\eta)d\eta}\,dY\\
&=\int_\G\int_{\g^\sharp}\! e^{i\<Y\mid\eta-\xi\>}\overline{f([-Y]\bu X,\eta)}d\eta\,dY.
\end{align}
\end{proof}

\begin{Corollary}\label{conzecinta}
For every $\Z\in\Xi$ one has ${\sf Op}(\varepsilon_\Z)^*={\sf Op}\big(\varepsilon_\Z^\#\big)$\,, with
\begin{equation*}\label{transpun}
\varepsilon_{(Z,\zeta)}^\#(X,\xi)=e^{i\<Z\mid\xi\>}e^{-i\<Z\bu X\mid\zeta\>}.
\end{equation*}
\end{Corollary}

\begin{proof}
This follows from \eqref{adjunct}, or by checking directly that ${\sf Op}\big(\varepsilon_\Z^\#\big)={\sf E}(\Z)^*$ for every $\Z\in\Xi$\,.
\end{proof}

One already has the linear topological isomorphism of Gelfand triples
\begin{equation*}\label{stupefict}
\begin{diagram}
\node{\S(\Xi)}\arrow{s,l}{\sf Op}\arrow{e,t}{}\node{L^2(\Xi)}\arrow{e,t}{} \arrow{s,l}{\sf Op}\node{\S'(\Xi)}\arrow{s,l}{{\sf Op}}\\ 
\node{\mathbb L\big[\S'(\g),\S(\g)\big]}\arrow{e,t}{}\node{\mathbb B^2\big[L^2(\g)\big]}\arrow{e,t}{}\node{\mathbb L\big[\S(\g),\S'(\g)\big]}
\end{diagram}
\end{equation*}
The horizontal arrows are linear continuous dense embeddings. The first vertical arrow is also an isomorphism of $^*$-algebras. Taking into account the fact that $\mathbb B^2\big[L^2(\g)\big]$ is a $H^*$-algebra (i.e. a complete Hilbert algebra) with respect to the operator product, the usual adjoint and the scalar product associated to the trace, one gets easily

\begin{Lemma}\label{hilbalg}
\begin{enumerate}
\item[(i)] 
$\big(L^2(\Xi),\#\,,^\#,\<\cdot,\cdot\>_{L^2(\Xi)}\big)$ is a $H^*$-algebra.
\item[(ii)]
$\big(\S(\Xi),\#,^\#\<\cdot,\cdot\>_{L^2(\Xi)}\big)$ is a Hilbert algebra.
\end{enumerate}
\end{Lemma}

In particular, this means that for every $f,g,h\in L^2(\Xi)$ one has
\begin{equation}\label{kiudata}
\<f\#g,h\>_{(\Xi)}=\big\<f,h\#g^\#\big\>_{(\Xi)}=\big\<g,f^\#\#h\big\>_{(\Xi)}\,,\ \ \<f,g\>_{(\Xi)}=\big\<g^\#\!,f^\#\big\>_{(\Xi)}\,.
\end{equation}

This allows a series of extensions by duality. By capital letters we denote distributions. We are going to skip the easy justifications and refer to \cite{MP} for an abstract approach.

\smallskip
First one extends
\begin{equation*}\label{lanceput}
\S'(\Xi)\times\S(\Xi)\overset{\#}{\longrightarrow}\S'(\Xi)\,,\quad\<F\#g,h\>_{(\Xi)}:=\big\<F,h\#g^\#\big\>_{(\Xi)}\,,
\end{equation*}
\begin{equation*}\label{lasfarsit}
\S(\Xi)\times\S'(\Xi)\overset{\#}{\longrightarrow}\S'(\Xi)\,,\quad\<g\#F,h\>_{(\Xi)}:=\big\<F,g^\#\#h\big\>_{(\Xi)}\,.
\end{equation*}

\begin{Definition}\label{janski}
{\it The Moyal algebra} is 
\begin{equation*}\label{moyal} 
\mathcal M(\Xi):=\big\{F\in\S'(\Xi)\!\mid\!F\#\S(\Xi)\subset\S(\Xi)\,,\ \S(\Xi)\#F\subset\S(\Xi)\big\}\,.
\end{equation*}
\end{Definition}

It is clear that this is a unital $^*$-algebra with
\begin{equation*}\label{lamijloc}
\<F\#G,h\>_{(\Xi)}:=\big\<F,h\#G^\#\big\>_{(\Xi)}\quad{\rm and}\quad\<F^\#\!,h\>_{(\Xi)}:=\big\<h^\#\!,F\big\>_{(\Xi)}\,,\quad\forall\,h\in\S(\Xi)\,;
\end{equation*}
the unit is the constant function $1$\,. Actually, by construction, it is the largest $^*$-algebra in which $\S(\Xi)$ is an essential bi-sided self-adjoint ideal. Anyhow, one has 
\begin{equation*}\label{bilat}
\mathcal M(\Xi)\#\mathcal S(\Xi)\#\mathcal M(\Xi)\subset\mathcal S(\Xi)\,.
\end{equation*}
We do not intend to discuss its natural topological structure. There is also an obvious way to get extensions
\begin{equation}\label{lancapat}
\S'(\Xi)\times\mathcal M(\Xi)\overset{\#}{\longrightarrow}\S'(\Xi)\,,\quad\mathcal M(\Xi)\times\mathcal S'(\Xi)\overset{\#}{\longrightarrow}\S'(\Xi)\,.
\end{equation}

By inspecting the definitions one realizes that

\begin{Proposition}\label{enfinn}
The pseudo-differential calculus extends to an isomorphism ${\sf Op}:\mathcal M(\Xi)\to\mathbb L[\S(\g)]\cap\mathbb L[\S'(\g)]$\,.
\end{Proposition}

Finally, let us introduce a symbol version of the $C^*$-algebra of all the bounded linear operators in $\H=L^2(\g)$\,, by pulling back structure through ${\sf Op}$\,.

\begin{Definition}\label{ru}
{\it The symbol $C^*$-algebra} is 
\begin{equation*}\label{ruj} 
\mathcal A(\Xi):=\{F\in\S'(\Xi)\!\mid\!{\sf Op}(F)\in\mathbb B(\H)\}\,,\quad\p\!F\!\p_{\mathcal A(\Xi)}\,:=\,\p\!{\sf Op}(F)\!\p_{\mathbb B(\H)}.
\end{equation*}
\end{Definition}

\begin{Remark}\label{sistand}
{\rm Obviously
\begin{equation*}\label{siostrand}
\S(\Xi)\subset L^2(\Xi)\subset\mathcal A(\Xi)\subset\S'(\Xi)\,,
\end{equation*}
and all the inclusions are strict. The exponential functions $\{\varepsilon_\Z\!\mid\!\Z\in\Xi\}$ and the constant functions are all in $[\mathcal A(\Xi)\cap\mathcal M(\Xi)]\!\setminus\!L^2(\Xi)$\,. There is no inclusion between $\mathcal A(\Xi)$ and $\mathcal M(\Xi)$\,.  Since by quantizing symbols only depending on $X\in\g$ one gets multiplication operators, it is clear that $L^\infty(\g)\subset\mathcal A(\Xi)$ and $C^\infty_{\rm pol}(\g)\subset\mathcal M(\Xi)$\,. Thinking of symbols only depending on  $\xi\in\g^\sharp$, yielding convolution operators by the inverse Fourier transform of the symbol, one gets other results. Since both $\mathcal A(\Xi)$ and $\mathcal M(\Xi)$ are $^*$-algebras, one can generate new examples by performing $\#$-products.  
}
\end{Remark}

%-------------------------------------------------------------------------------------------------------
\section{Phase-space shifts}\label{frikasse}
%----------------------------------------------------------------------------------------------------

We introduce at the symbol level the analog of Definition \ref{niulaif}.

\begin{Definition}\label{intrinsick}
For every $\Y,\Z\in\Xi$ and $f\in\S'(\Xi)$ we set $\th_{\Y,\Z}(f):=\varepsilon_\Y\# f\#\varepsilon_\Z^\#$\,.
\end{Definition}

In fact, having in view the properties of the functions $\varepsilon_\Z$\,, we see that the mapping $\th_{\Y,\Z}$ acts in any of the spaces $\S(\Xi),L^2(\Xi),\A(\Xi),\mathcal M(\Xi),\S'(\Xi)$\,. We also refer to Remark \ref{besides}. 

\begin{Proposition}\label{tragconsec}
One has
\begin{equation*}\label{theta}
\Th_{\Y,\Z}[{\sf Op}(f)]={\sf Op}\big[\th_{\Y,\Z}(f)\big]\,.
\end{equation*}
The $C^*$-algebra $\A(\Xi)$ is invariant under all the mappings $\th_{\Y,\Z}$\,.
\end{Proposition}

\begin{proof}
We have
\begin{align}
\Th_{\Y,\Z}[{\sf Op}(f)]&={\sf E}(\Y){\sf Op}(f){\sf E}(\Z)^*\!={\sf Op}(\varepsilon_\Y){\sf Op}(f){\sf Op}(\varepsilon_\Z^\#)\\
&={\sf Op}\big(\varepsilon_\Y\# f\#\varepsilon_\Z^\#\big)={\sf Op}\big[\th_{\Y,\Z}(f)\big]\,.
\end{align}
Invariance follows from this, since at an operator level $\A(\Xi)$ is $\mathbb B(\H)$\,, which is left invariant by multiplying to the left and to the right with elements of the Weyl system.
\end{proof}

We are going to need below the following result:

\begin{Lemma}\label{bucluk}
For every $u,v\in L^2(\Xi)$ and $\Y,\Z\in\Xi$ one has
\begin{equation}\label{ciubuk}
\th_{\Y,\Z}\big(\mathcal E_{u,v}\big)=\mathcal E_{{\sf E}(\Y)u,{\sf E}(\Z)v}\,.
\end{equation}
\end{Lemma}

\begin{proof}
One can write
\begin{align}
{\sf Op}\big[\th_{\Y,\Z}\big(\mathcal E_{u,v}\big)\big]&=\Th_{\Y,\Z}\big[{\sf Op}\big(\mathcal E_{u,v}\big)\big]={\sf E}(\Y)\big[\<\cdot,v\>u\big]{\sf E}(\Z)^*\\
&=\<\cdot,{\sf E}(\Z)v\>{\sf E}(\Y)u={\sf Op}\big(\mathcal E_{{\sf E}(\Y)u,{\sf E}(\Z)v}\big)\,,
\end{align}
which implies \eqref{ciubuk}.
\end{proof}

\begin{Remark}\label{cazzul}
The explicit form of $\th_{\Y,\Z}(f)$ is less important than the way it has been constructed, and will not be used here. However, for convenience, we are going to record the diagonal case $\th_{\Z,\Z}\equiv\th_\Z$ (forming a family of automorphisms).  One of the reasons is that it leads to the covariant symbol of the operators ${\sf Op}(f)$\,; see \cite{M}. By a direct computation one gets
\begin{equation}\label{heta}
\big[\th_{(Z,\zeta)}(f)\big](X,\xi)=\!\int_\g\!\int_{\g^\sharp}\! e^{i\<X\bu[-Y]\mid\eta-\xi\>}e^{i\<X-Y\mid\zeta\>}f\Big([-Z]\bu X,\widetilde{\sf Ad}^\sharp_{-Z}(\eta)\Big)dYd\eta
\end{equation}
in terms of the coadjoint action \eqref{faras}. If $\G=\R^n$ the coadjoint action is trivial and $X\bu[-Y]=X-Y$\,, so one gets 
\begin{equation*}\label{bigineq}
\big[\th_{(Z,\zeta)}(f)\big](X,\xi)=f(X-Z,\xi-\zeta)\,,
\end{equation*}
implying that the phase-space translations of the symbols are implemented, at the level of the quantization, by conjugations with the Weyl system.
\end{Remark}

\begin{Remark}\label{buluc}
One can use the automorphism family $\big\{\th_\Z\!\mid\!\Z\in\Xi\big\}$ to define a sophisticated form of convolution, that we intend to use in a future publication. For suitable functions $\varphi,f$ on $\Xi$ we write
\begin{equation}\label{clubsuite}
\varphi\!\ddagger_\theta\! f=\int_\Xi\!\varphi(\Z)\theta_\Z(f)\,d\Z=\int_\Xi \varphi(\Z)\,\varepsilon_\Z\#f\#\varepsilon_\Z^\#\,d\Z\,.
\end{equation}
For $\G=\R^n$ this boils down to the usual additive convolution, since $\th_\Z$ reduces to a translation.
One gets easily an explicit formula:
\begin{equation*}\label{sprelumina}
\big(\varphi\!\ddagger_\theta\! f\big)(\X)=\int_\Xi \phi_\X(\Z)f(\Z)d\Z\,,
\end{equation*}
where
\begin{align}\label{maivrei}
\phi_{(X,\xi)}(Z,\zeta):=\int_\g\!\int_{\g^\sharp} &e^{-i\<X\bu[-N]\bu Z\bu[-X]\mid\xi\>}e^{i\<X-X\bu[-Z]\bu N\mid\mu\>}\\
&e^{i\<Z\bu[-N]\mid\zeta\>}\varphi(X\bu[-Z],\mu)dNd\mu\,.
\end{align}
It is easy to verify that, for $\G=\R^n$, one gets $\phi_{(X,\xi)}(Z,\zeta)=\varphi(X-Z,\xi-\zeta)$\,.

\smallskip
As with the usual convolution, setting
\begin{equation*}\label{frampton}
\varphi_t(\Z):=t^{-2n}\varphi\big(t^{-1}\Z\big)\,,\quad t>0\,,\ \Z\in\Xi\,,\ \varphi\in\S(\Xi).
\end{equation*}
one gets $\varphi_t\!\ddagger_\th\!f\underset{t\to 0}{\longrightarrow} f$ pointwise if $f$ is bounded and continuous.
\end{Remark}

%-------------------------------------------------------------------------------------------------------
\section{Coorbit spaces - a short overview}\label{frikamie}
%----------------------------------------------------------------------------------------------------

Let us pick a normalized "window" (or "atom") $w$ belonging to the Fr\'echet space $\S(\g)\hookrightarrow L^2(\g)$\,.
In terms of the Fourier-Wigner transform, the linear mapping 
\begin{equation*}\label{balaur}
\mathcal E_{w}:\S'(\g)\rightarrow\S'(\g\times\g^\sharp)\,,\quad\mathcal E_{w}(u):=\mathcal E_{u,w}
\end{equation*} 
will be used to pull back algebraic and topological structures. It is isometric from $L^2(\g)$ to $L^2(\g\times\g^\sharp)$ and this has standard consequences (inversion and reproduction formulae). Let us record the explicit form of its adjoint
\begin{equation}\label{adhuntul}
\mathcal E_{w}^\dag(h)=\int_\Xi h(\X){\sf E}(\X)^*w\,d\X.
\end{equation}

\begin{Definition}\label{carcaluc}
Let $(\mathcal B,\p\cdot\p_\mathcal B)$ be a normed space continuously embedded in $\S'(\g\times\g^\sharp)$\,. Its {\it coorbit space (associated to the window $w$)} is
\begin{equation}\label{paianjenel}
{\sf co}_{w}(\mathcal B):=\big\{u\in\S'(\g)\mid \mathcal E_{w}(u)\in\mathcal B\big\}
\end{equation}
with the norm $\p\!u\!\p_{{\sf co}_{w}(\mathcal B)}\,:=\,\p\!\mathcal E_{w}(u)\!\p_\mathcal B$\,. 
\end{Definition}

If $\mathcal B$ is just a vector subspace, we still use \eqref{paianjenel} to define a subspace of $\S'(\g)$\,. The case of a locally convex space $\mathcal B$ is also important. 
%It is also possible to start with an incomplete norm $\p\!\cdot\!\p_\mathcal B$ on a vector subspace $\mathcal B_0$\,, to induce the norm $\p\!u\!\p_{{\sf co}_{w}(\mathcal B)}$ as above 
%and then to complete getting a (coorbit) Banach space.

\begin{Remark}\label{sarpe}
Developing the abstact part of the theory of the spaces ${\sf co}_{w}(\mathcal B)$ is quite standard, relying on the good properties of the isometry $\mathcal E_{w}$\,, and it will not be done here. Let us just state that if $\mathcal B\hookrightarrow\S'(\g\times\g^\sharp)$ is Banach, then ${\sf co}_{w}(\mathcal B)$ is a Banach space continuously embedded in $\S'(\g)$\,. Simple arguments based on the inversion formula and the mapping properties of $\mathcal E_{w}^\dag$ show that 
$$
{\sf co}_{w}\big[L^2(\g\times\g^\sharp)\big]=\H\,,\quad{\sf co}_{w}\big[\S(\g\times\g^\sharp)\big]=\S(\g)\,,\quad{\sf co}_{w}\big[\S'(\g\times\g^\sharp)\big]=\S'(\g)\,.
$$
\end{Remark}

Weighted mixed $L^{pq}$-spaces of functions defined on $\g\times\g^\sharp$ are nice examples of spaces $\mathcal B$ to start with. For the case of the Heisenberg group, a comparison with \cite{FRR} would be interesting.

\smallskip
To define coorbit spaces of functions on $\Xi=\g\times\g^\sharp$, that could play the role of symbols of the ${\sf Op}$-calculus, one needs a good map (at least an isometry) transforming functions/distributions on $\Xi$ into functions/distributions on $\Xi\times\Xi$\,. One solution is to proceed by analogy, defining in $\Xi\times\Xi\cong(\g\times\g)\times\big(\g^\sharp\times\g^\sharp\big)$ a Weyl system and a Fourier-Wigner transform, doubling the number of variables. We did not check, but a starting point could be the Weyl system
$$
\big[\mathbf E\big((Y,\eta),(Z,\zeta)\big)h\big](X,\xi)):=e^{i\<Z,\xi\>}e^{-i\<X,\zeta\>}h\big([-Y]\bu X,\xi-\eta\big)\,,
$$
attaching unitary operators in $L^2(\Xi)$ to points in $\Xi\times\Xi$\,. The non-commutative group structure of $\Xi=\g\times\g^\sharp$ has been taken into account. 

\smallskip
It seems that nobody has done this for the case of a nilpotent Lie algebra $\g$\,, but is is likely that this can be done, and connecting the coorbit spaces on $\g$ and on $\Xi$\,, respectively, via the pseudo-differential calculus, would be successful. We will sketch a different approach, relying on the previously defined phase-space shifts and having some connections with ideas from \cite{M0}.

\begin{Definition}\label{minte}
Let $h\in\S(\Xi)\setminus\{0\}$ (most often $\p\!h\!\p_{(\Xi)}\,=1$)\,. One defines for $f\in\S'(\Xi)$ and $\Y,\Z\in\Xi$
\begin{equation}\label{disperare}
\big[\mathfrak E_h(f)\big](\Y,\Z):=\big\<\th_{\Y,\Z}(f),h\big\>_{\!(\Xi)}=\big\<\varepsilon_\Y\#f\#\varepsilon_\Z^\#,h\big\>_{\!(\Xi)}\,.
\end{equation}
\end{Definition}

Note that, by \eqref{lancapat}, one has 
$$
\varepsilon_\Y\#f\#\varepsilon_\Z^\#\in\mathcal M(\Xi)\#\S'(\Xi)\#\mathcal M(\Xi)\subset\S'(\Xi)\,,
$$ 
so by our choice $h\in\S(\Xi)$ the expression \eqref{disperare} makes sense. If $f\in L^2(\Xi)$\,, one has 
$$
\varepsilon_\Y\#f\#\varepsilon_\Z^\#\!\in\A(\Xi)\# L^2(\Xi)\#\A(\Xi)\subset L^2(\Xi)
$$ 
($L^2(\Xi)$ is an ideal in $\A(\Xi)$\,, since it corresponds to Hilbert-Schmidt operators) and an $L^2$-window is directly available. Other situations can be accommodated. 

\begin{Proposition}\label{lafin}
For every $h,f\in L^2(\Xi)$ one has 
\begin{equation}\label{finfin}
\big\Vert\,\mathfrak E_h(f)\,\big\Vert_{(\Xi\times\Xi)}=\,\p\!h\!\p_{(\Xi)}\,\p\!f\!\p_{(\Xi)}.
\end{equation}
\end{Proposition}

\begin{proof}
The family $\big\{\mathcal E_{u,v}\!\mid\! u,v\in L^2(\G)\big\}$ is total in $L^2(\Xi)$\,, since by the ${\sf Op}$-quantization it yields all the rank one operators (see Proposition \ref{cateva},\,(iii)), forming a total family in $\mathbb B^2\big[L^2(\g)\big]$\,. One has
\begin{align}
\big[\mathfrak E_{\mathcal E_{u,v}}\!(\mathcal E_{u',v'})\big](\X,\Y)&=\big\<\varepsilon_\X\#\mathcal E_{u',v'}\#\varepsilon_\Y^\#,\mathcal E_{u,v}\big\>_{(\Xi)}\\
&\overset{\eqref{ciubuk}}{=}\big\<\mathcal E_{{\sf E}(\X)u',{\sf E}(\Y)v'},\mathcal E_{u,v}\big\>_{(\Xi)}\\
&\overset{\eqref{orthog}}{=}\<{\sf E}(\X)u',u\>_{(\g)}\<v,{\sf E}(\Y)v'\>_{(\g)}\\
&=\mathcal E_{u',u}(\X)\overline{\mathcal E_{v',v}(\Y)}\,,
\end{align}
meaning that $\mathfrak E_{\mathcal E_{u,v}}\!(\mathcal E_{u',v'})=\mathcal E_{u',u}\otimes\overline{\mathcal E_{v',v}}$\,. Then using once again the orthogonal relations \eqref{orthog} leads easily to the result (work with scalar products and then take diagonal values to get \eqref{finfin}).
\end{proof}

\begin{Definition}\label{garcaluc}
Let $(\mathfrak B,\p\cdot\p_\mathfrak B)$ be a normed space continuously embedded in $\S'(\Xi\times\Xi)$\,. Its {\it coorbit space associated to the normalized window $h\in\S(\Xi)$} is
\begin{equation}\label{baianjenel}
{\sf Co}_{h}(\mathfrak B):=\big\{f\in\S'(\Xi)\mid \mathfrak E_{h}(f)\in\mathfrak B\big\}
\end{equation}
with the norm $\p\!u\!\p_{{\sf Co}_{h}(\mathfrak B)}\,:=\,\p\!\mathfrak E_{h}(f)\!\p_\mathfrak B$\,. If $\mathfrak B$ is just a subspace, we still define the vector space \eqref{baianjenel}, but without norm.
\end{Definition}

Recall the algebraic rules of (suitable) integral kernels $K,L:\Xi\times\Xi\to\mathbb C$\,:
\begin{equation*}\label{integrale}
(K\circ L)(\Y,\Z):=\int_\Xi K(\Y,\X)L(\X,\Z)d\X\,,\quad K^\circ(\Y,\Z):=\overline{K(\Z,\Y)}\,.
\end{equation*}
Of course, they are ment to emulate the multiplication and the adjoint of integral operators.

\begin{Proposition}\label{fulminant}
If $\,h,k\in\S(\Xi)$\,, then for every $f,g\in L^2(\Xi)$ we have
\begin{equation*}\label{superalgebric}
\mathfrak E_{h\#k}(f\#g)=\mathfrak E_h(f)\circ\mathfrak E_k(g)\quad{\rm and}\quad\mathfrak E_{h^\#}\big(f^\#\big)=\mathfrak E_h(f)^\circ\,.
\end{equation*}
\end{Proposition}

\begin{proof}
One uses the rules \eqref{kiudata} of a Hilbert algebra and the relation \eqref{vormula}, getting
\begin{align}
\big[\mathfrak E_h(f)\circ\mathfrak E_k(g)\big](\Y,\Z)&=\int_\Xi \mathfrak E_h(f)(\Y,\X)\,\mathfrak E_k(g)(\X,\Z)\,d\X\\
&=\int_\Xi \big\<\varepsilon_\Y\#f\#\varepsilon_\X^\#,h\big\>_{\!(\Xi)}\big\<\varepsilon_\X\#g\#\varepsilon_\Z^\#,k\big\>_{\!(\Xi)}\,d\X\\
&=\int_\Xi \big\<h^\#\#\varepsilon_\Y\#f,\varepsilon_\X\big\>_{\!(\Xi)}\big\<\varepsilon_\X,k\#\varepsilon_\Z\#g^\#\big\>_{\!(\Xi)}\,d\X\\
&=\big\<h^\#\#\varepsilon_\Y\#f,k\#\varepsilon_\Z\#g^\#\big\>_{\!(\Xi)}=\big\<h\#h\#\varepsilon_\Y\#f\#g,\varepsilon_\Z\big\>_{\!(\Xi)}\\
&=\big\<\varepsilon_\Y\#f\#g\#\varepsilon_\Z^\#,h\#k\big\>_{\!(\Xi)}=\mathfrak E_{h\#k}(f\#g)(\Y,\Z)\,.
\end{align}
and
\begin{align}
\mathfrak E_h(f)^\circ(\Y,\Z)&=\overline{\mathfrak E_h(f)(\Z,\Y)}=\overline{\big\<\varepsilon_\Z\#f\#\varepsilon_\Y^\#,h\big\>_{\!(\Xi)}}\\
&=\big\<h,\varepsilon_\Z\#f\#\varepsilon_\Y^\#\big\>_{\!(\Xi)}=\big\<f^\#\#\varepsilon_\Z^\#,\varepsilon_\Y^\#\#h^\#\big\>_{\!(\Xi)}\\
&=\big\<\varepsilon_\Y\#f^\#\#\varepsilon_\Z^\#,h^\#\big\>_{\!(\Xi)}=\mathfrak E_{h^\#}\big(f^\#\big)(\Y,\Z)\,.
\end{align}
\end{proof}

\begin{Corollary}\label{infekt}
If the vector space $\mathfrak B$ is an involutive algebra with respect to $\big(\circ,^\circ\big)$ and $h=h\#h=h^\#$, then ${\sf Co}_{h}(\mathfrak B)$ is an involutive algebra with respect to $\big(\#,^\#\big)$\,. 
\end{Corollary}

\begin{proof}
Suppose that $f,g\in{\sf Co}_{h}(\mathfrak B)$\,, which means that $\mathfrak E_h(f),\mathfrak E_h(g)\in\mathfrak B$\,. Then
\begin{equation*}\label{efort}
\mathfrak E_h(f\#g)=\mathfrak E_{h\#h}(f\#g)=\mathfrak E_h(f)\circ\mathfrak E_h(g)\in\mathfrak B\,,
\end{equation*}
implying that $f\# g\in{\sf Co}_{h}(\mathfrak B)$\,. Invariance under the involution $^\#$ is checked similarly, using the self-adjointness of the window $h$\,.
\end{proof}

\begin{Remark}\label{c+norm}
In the framework of the Corollary, if $\p\!\cdot\p_\mathfrak B$ is a $C^*$-norm on $\big(\mathfrak B,\circ,^\circ\big)$\,, then $\p\!\cdot\p_{{\sf Co}_h(\mathcal B)}$ is a $C^*$-norm on $\big({\sf Co}_h(\mathcal B),\#,^\#\big)$:
$$
\p\!g^\#\#g\!\p_{{\sf Co}_h(\mathcal B)}\,=\big\Vert\,\mathfrak E_h(g^\#\#g)\,\big\Vert_{\mathfrak B}=\big\Vert\,\mathfrak E_h(g)^\circ\!\circ\mathfrak E_h(g)\,\big\Vert_{\mathfrak B}=\big\Vert\,\mathfrak E_h(g)\,\big\Vert_{\mathfrak B}^2=\,\p\!g\!\p_{{\sf Co}_h(\mathcal B)}^2.
$$
\end{Remark}

We describe now an abstract situation in which pseudo-differential operators with symbols in a coorbit space (of symbols) are well-defined and bounded between two coorbit spaces of vectors. Note that this requires a correlation of the windows; the Wigner transform $\mathcal W$ has been introduced in \eqref{napasttaka}.

\begin{Theorem}\label{margisnitele}
Let $w_1,w_2\in\S(\g)$ with $\p\!w_1\!\p_{(\g)}\,=\,\p\!w_2\!\p_{(\g)}\,=\!1$\,. Let $\p\!\cdot\!\p_{\mathcal B_1}$ and $\p\!\cdot\!\p_{\mathcal B_2}$ two norms on $\S(\Xi)$ 
and  $\p\!\cdot\!\p_{\mathfrak B}$ a norm on $\S(\Xi\times\Xi)$\,. 

Suppose that for every $\Psi\in\S(\Xi\times\Xi)$ the integral operator ${\sf Int}(\Psi)$ is bounded from $\big(\S(\Xi),\p\!\cdot\!\p_{\mathcal B_1}\!\!\big)$ to $\big(\S(\Xi),\p\!\cdot\!\p_{\mathcal B_2}\!\!\big)$\,, with operatorial norm less or equal than $C\!\p\!\Psi\!\p_{\mathfrak B}\,$ for some positive absolute constant $C$. 

Then, for every $f\in\S(\Xi)$\,, the pseudo-differential operator $\,{\sf Op}(f)$ is bounded from $\big(\S(\g),\p\!\cdot\!\p_{{\sf co}_{w_1}\!(\mathcal B_1)}\!\!\big)$ to $\big(\S(\g),\p\!\cdot\!\p_{{\sf co}_{w_2}\!(\mathcal B_2)}\!\!\big)$\,, with operatorial norm less or equal than $C\!\p\!f\!\p_{\,{\sf Co}_{\mathcal W_{w_1,w_2}}\!(\mathfrak B)}$\,.
\end{Theorem}

\begin{proof}
We first show that for $w_1,w_2\in\S(\g)$ and $f\in\S(\Xi)$\,, in terms of the Wigner transform \eqref{napasttaka}, one has
\begin{equation}\label{formmula}
\mathcal E_{w_2}{\sf Op}(f)\mathcal E_{w_1}^\dag={\sf Int}\big[\mathfrak E_{\mathcal W_{w_1,w_2}}(f)\big]\,.
\end{equation}
For this we compute
\begin{align}
\big(\mathcal E_{w_2}{\sf Op}(f)\mathcal E_{w_1}^\dag g\big)(\X)&=\Big\<{\sf E}(\X){\sf Op}(f)\mathcal E^\dag_{w_1}g,w_2\Big\>\\
&\overset{\eqref{adhuntul}}{=}\Big\<{\sf E}(\X){\sf Op}(f)\int_\Xi g(\Y){\sf E}(\Y)^*d\Y\,w_1,w_2\Big\>\\
&=\int_\Xi g(\Y)\big\<{\sf Op}\big(\varepsilon_\X\#f\#\varepsilon_\Y^\#\big)w_1,w_2\big\>\,d\Y\\
&\overset{\eqref{napasttaka}}{=}\int_\Xi \big\<\varepsilon_\X\#f\#\varepsilon_\Y^\#,\mathcal W_{w_1,w_2}\big\>_{\!(\Xi)}\,g(\Y)\,d\Y\\
&=\big({\sf Int}\big[\mathfrak E_{\mathcal W_{w_1,w_2}}\!(f)\big]g\big)(\X)\,.
\end{align}

Then the norm estimate is easy: for every $u\in\S(\g)$
\begin{align}
\p\!{\sf Op}(f)u\!\p_{{\sf co}_{w_2}\!(\mathcal B_2)}&=\,\p\!\mathcal E_{w_2}{\sf Op}(f)u\!\p_{\mathcal B_2}=\big\Vert\,{\sf Int}\big[\mathfrak E_{\mathcal W_{w_1,w_2}}\!(f)\big]\mathcal E_{w_1}(u)\,\big\Vert_{\mathcal B_1}\\
&\le \big\Vert\,{\sf Int}\big[\mathfrak E_{\mathcal W_{w_1,w_2}}\!(f)\big]\big\Vert_{\mathbb B(\mathcal B_1,\mathcal B_2)}\p\!\mathcal E_{w_1}(u)\!\p_{\mathcal B_1}\\
&\le C\,\big\Vert\,\mathfrak E_{\mathcal W_{w_1,w_2}}\!(f)\,\big\Vert_{\mathfrak B}\p\!\mathcal E_{w_1}(u)\!\p_{\mathcal B_1}\\
&= C\,\big\Vert\,f\,\big\Vert_{{\sf Co}_{\mathcal W_{w_1,w_2}}\!(\mathfrak B)}\!\p\!u\!\p_{{\sf co}_{w_1}\!(\mathcal B_1)}.
\end{align}
\end{proof}

\begin{Remark}\label{critica}
Theorem \ref{margisnitele} provides a boundedness result for pseudo-differential operators involving coorbit norms both at the level of vectors and at the level of symbols. However, in the statement and the proof of this Theorem, the initial or the induced norms are only defined on (various) Schwartz spaces and the action of the operators are also confined to such spaces. Of course, automatically, there are bounded extensions to the corresponding completions. But, for a really nice result, some technical issues still have to be solved. For example, if $\mathcal B_1$ denotes the completion of $\big(\S(\Xi),\p\!\cdot\!\p_{\mathcal B_1}\!\!\big)$, is it true that the (very relevant) completion of $\big(\S(\g),\p\!\cdot\!\p_{{\sf co}_{w_1}\!(\mathcal B_1)}\!\!\big)$ may be identified with the coorbit space ${\sf co}_{w_1}\!(\mathcal B_1)$\,? There is a similar question starting with the completion $\mathfrak B$ of $\big(\S(\Xi\times\Xi),\p\!\cdot\!\p_{\mathfrak B}\!\!\big)$\,. In addition, one would like to treat boundedness for coorbit spaces associated to spaces of temperate distributions in which the Schwartz space is not dense. Besides this, many other topics deserves attention, as duality, interpolation, equivalent norms, dependence of windows, decompositions, Schatten-von Neumann behavior, etc. They will be treated systematically in a subsequent publication.
\end{Remark}

\begin{Remark}\label{conkrete}
Another reason to invest effort in a future article is {\it concreteness}. Besides abstract results, valid for general choices, many interesting facts will only occur in particular situations. Even if $\g=\R^n$ (Abelian), most of the previous work has been dedicated to {\it weighted modulation spaces}, having mixed $L^{p,q}$-spaces as a starting point. In addition, an important and difficult issue is to compare the modulation (or the coorbit) spaces with other function spaces, defined by different techniques. Frames should also be studied. Hopefully, more specific Lie algebra features will appear at a certain moment.
\end{Remark}

{\textbf Acknowledgement.}
The author has been supported by the Fondecyt Project 1160359. 

\smallskip
He is grateful for having the opportunity to participate in the Conference MicroLocal and Time Frequency Analysis 2018 in honor of Luigi Rodino on the occasion of his 70th Birthday.

%-------------------------------------------------------------------------------------------------------

\end{document}